\documentclass[12pt]{amsart}
\usepackage{fullpage}

\usepackage{graphicx}				

\usepackage{amssymb}
\usepackage{amsthm}
\usepackage{amsmath}
\usepackage{cite}
\usepackage[all]{xy}
\usepackage{tikz}
\usepackage{enumerate}
\usepackage{verbatim}
\usepackage{mathtools}
\usepackage{hyperref}
\usepackage{colonequals}
\usepackage{mathrsfs}
\usepackage{young}

\usetikzlibrary{decorations.markings}

\theoremstyle{plane}
\newtheorem{theorem}{Theorem}[]
\newtheorem{lemma}[theorem]{Lemma}

\newtheorem{cor}[theorem]{Corollary}

\theoremstyle{definition}
\newtheorem{definition}[theorem]{Definition}
\newtheorem*{ack}{Acknowledgement}

\newcommand{\F}{\mathbb{F}}

\let\k\relax
\newcommand{\k}{\mathbf{k}}

\newcommand{\B}{\mathcal{B}}
\newcommand{\E}{\mathcal{E}}
\newcommand{\U}{\mathcal{U}}

\newcommand{\D}{\mathcal{D}}
\newcommand{\OO}{\mathcal{O}}

\newcommand{\X}{\mathcal{X}}
\newcommand{\Y}{\mathcal{Y}}
\newcommand{\V}{\mathcal{V}}
\let\L\relax
\newcommand{\L}{\mathcal{L}}
\let\P\relax
\newcommand{\P}{\mathcal{P}}

\newcommand{\n}{\mathfrak{n}}

\newcommand{\g}{\mathfrak{g}}

\let\b\relax
\newcommand{\b}{\mathfrak{b}}
\let\t\relax
\newcommand{\t}{\mathfrak{t}}

\newcommand{\M}{\mathbb{M}}
\newcommand{\N}{\mathbb{N}}
\newcommand{\UU}{\mathbb{U}}
\newcommand{\VV}{\mathbb{V}}

\newcommand{\e}{\varepsilon}

\DeclarePairedDelimiter\ceil{\lceil}{\rceil}
\DeclarePairedDelimiter\floor{\lfloor}{\rfloor}

\DeclareMathOperator{\Ind}{Ind}

\newcommand{\ola}[1]{\overleftarrow{#1}}

\newcommand{\pol}[1]{\mathbb{#1}}

\title{A Comparison of Two Complexes}
\author{Dongkwan Kim}
\address{Department of Mathematics\\
  Massachusetts Institute of Technology\\
  Cambridge, MA 02139-4307\\
  U.S.A.}
\email{sylvaner@math.mit.edu}
\date{\today}							

\begin{document}
\begin{abstract}
In this paper we prove the conjecture of Lusztig in \cite[Section 4]{charring}. Given a reductive group over $\overline{\mathbb{F}_q}[\e]/(\e^r)$ for some $r \geq 2$, there is a notion of a character sheaf defined in \cite[Section 8]{chargen}. On the other hand, there is also a geometric analogue of the character constructed by G\'erardin \cite{gerardin}. The conjecture in \cite[Section 4]{charring} states that the two constructions are equivalent, which Lusztig also proved for $r=2, 3, 4$. Here we generalize his method to prove this conjecture for general $r$. As a corollary we prove that the characters derived from these two complexes are equal.
\end{abstract}

\maketitle

\tableofcontents

\section{Introduction}
The purpose of this paper is to prove the conjecture in \cite[Section 4]{charring}. We first recall its setting. Let $\k$ be an algebraic closure of a finite field $\mathbb{F}_q$ where $q$ is a power of prime $p$ and $G$ be a connected reductive algebraic group over $\k$. Let $T$ be a maximal torus and $B$ a Borel subgroup of $G$ which contains $T$. Also we let $U$ be the unipotent radical of $B$. We denote the Lie algebra of $G, B, T, U$ by $\g, \b, \t, \mathfrak{u}$, respectively.

For a fixed integer $r \geq 2$, we define $G_r \colonequals G(\k[\e]/(\e^r))$ where $\e$ is an indeterminate. Note that it has a natural algebraic group structure over $\k$. Similarly, we define $B_r \colonequals B(\k[\e]/(\e^r)), T_r \colonequals T(\k[\e]/(\e^r)),$ and $U_r\colonequals U(\k[\e]/(\e^r))$, which are again considered as algebraic groups over $\k$. 

Throughout this paper, we assume $p \geq r$ and thus we have an isomorphism of varieties
\begin{gather}\label{giso}
G \times \g^{r-1} \xrightarrow{\simeq} G_r : (x, X_1, \cdots, X_{r-1}) \mapsto x e^{\e X_1} \cdots e^{\e^{r-1} X_{r-1}}.
\end{gather}
Note that $e^{\e X_1}, \cdots, e^{\e^{r-1} X_{r-1}}$ are well-defined since $p \geq r$. If $G$ is abelian (thus so is $\g$) then (\ref{giso}) is also an isomorphism of algebraic groups.

Now we define 
\begin{align*}
\tilde{G}_r =  & \{(B_rg, g') \in (B_r\backslash G_r)\times G_r  \mid gg'g^{-1} \in B_r\}
\end{align*}
and consider the following diagram
\begin{displaymath}
\xymatrix{
T_r&\tilde{G}_r \ar[l]_{\tilde{\tau}} \ar[r]^{\tilde{\pi}}&G_r }
\end{displaymath}
where the morphisms on the diagram are defined as follows.
\begin{gather*}
\tilde{\pi}(B_rg, g') = g',\qquad \tilde{\tau}(B_rg, g') = \sigma(gg'g^{-1})
\end{gather*}
Here $\sigma: B_r \rightarrow T_r$ is the composition of the quotient morphism $B_r \rightarrow B_r/U_r$ and the inverse of $T_r \hookrightarrow B_r \rightarrow B_r/U_r$.


For any variety $X$ over $\k$, we denote by $\D(X)$ the bounded constructible derived category of $\ell$-adic sheaves where $\ell \neq p$ is a fixed prime. One of the main goals in \cite{charring} is to show that for a generic character sheaf $\L$ on $T_r$, $\tilde{\pi}_!\tilde{\tau}^*\L$ is an intersection cohomology complex. (For the definition of character sheaves on $T_r$ one may refer to \cite{chargen}.) One obstacle for this problem is that unlike the case $r=1$, the morphism $\tilde{\pi} : \tilde{G}_r \rightarrow G_r$ is no longer proper.

In \cite{charring} one strategy is explained; first we compare this complex to a similar one given by a geometric analogue of the character in \cite{gerardin} and use Fourier-Deligne transform to show that the latter one is indeed an intersection cohomology complex up to shift. In this paper we are interested in comparing these two complexes, which is proven for $r=2,3,4$ in \cite[Section 4]{charring}. Here we generalize the method of \cite{charring} to arbitrary $r\geq 2$.

\begin{ack}
I thank George Lusztig for suggesting this topic, carefully reading the draft of this paper, and giving valuable feedback.
\end{ack}

\section{Notations and the main theorem}
We recall the isomorphism of varieties (\ref{giso}) on the previous section for $T_r$.
\begin{gather}\label{tiso}
T \times \t^{r-1} \xrightarrow{\simeq} T_r : (x, X_1, \cdots, X_{r-1}) \mapsto x e^{\e X_1} \cdots e^{\e^{r-1} X_{r-1}}
\end{gather}
Since $T$ is abelian, this is indeed an isomorphism of algebraic groups. From now on we fix a non-degenerate symmetric bilinear invariant form $\langle\ ,\ \rangle$ on $\g$ and a non-trivial additive character $\psi: \F_q \rightarrow \mathbb{C}^*$. We consider the following complex
\begin{displaymath}
\E \boxtimes \L_{f} \in \D(T\times \t^{r-1})
\end{displaymath}
where $\E$ is a character sheaf on $T$ and $\L_{f}$ is the Artin-Schreier sheaf with respect to $\psi$ and the function
\begin{displaymath}
f: \t^{r-1} \rightarrow \k : (X_1, \cdots, X_{r-1}) \mapsto  \sum_{i=1}^{r-1} \langle A_i, X_i \rangle
\end{displaymath}
where $A_1, \cdots, A_{r-1} \in \t$. (For the definition of Artin-Schreier sheaf, we refer readers to \cite[0.3]{charring}.) By the isomorphism (\ref{tiso}), we may consider it as a complex on $T_r$, i.e. $\E \boxtimes \L_f \in \D(T_r)$. 

\begin{definition}
We say that this complex $\E \boxtimes \L_f\in \D(T_r)$ is \emph{generic} if $A_{r-1}$ is regular semisimple.
\end{definition}


One of the main objects in \cite{charring} is the complex $\tilde{\pi}_! \tilde{\tau}^* (\E \boxtimes \L_{f}) \in \D(G_r)$. In order to compare this with a geometric analogue of the character in \cite{gerardin}, we first replace $\tilde{G}_r$ with a fiber bundle which we call $\X$ as follows.
\begin{displaymath}
\X \colonequals \{(Tg, g') \in (T\backslash G_r)\times G_r  \mid gg'g^{-1} \in B_r\}
\end{displaymath}
Then we have an obvious morphism $\rho: \X \rightarrow \tilde{G}_r : (Tg, g') \mapsto (B_rg, g')$, which makes $\X$ a fiber bundle with the fiber isomorphic to an affine space, or the unipotent radical of $B_r$.

We give another description of $\X$. For $j \geq 1$, let 
\begin{displaymath}
\pol{P}_j = u_j(\pol{A}_1, \cdots, \pol{A}_j,  \pol{B}_1, \cdots, \pol{B}_j, \pol{C}_1, \cdots, \pol{C}_j)
\end{displaymath}
be a non-commutative polynomial in indeterminates $\pol{A}_i, \pol{B}_i, \pol{C}_i$ for $i \geq 1$ such that the following equation holds.
\begin{align*}
e^{\e \pol{A}_1} \cdots e^{\e^j \pol{A}_j} e^{\e \pol{B}_1} \cdots e^{\e^j \pol{B}_j} (e^{\e \pol{C}_1} \cdots e^{\e^j \pol{C}_j})^{-1} = e^{\e \pol{P}_1} \cdots e^{\e^j \pol{P}_j} \mod \e^{j+1}
\end{align*}
By the Campbell-Hausdorff formula and induction on $j$, it is easy to show that $\pol{P}_j - (\pol{A}_j+\pol{B}_j-\pol{C}_j)$ is a Lie polynomial of $\pol{A}_1, \cdots, \pol{A}_{j-1}, \pol{B}_1, \cdots, \pol{B}_{j-1}, \pol{C}_1, \cdots, \pol{C}_{j-1}$ without a constant or linear term. Now we see that $\X$ may be also defined by the following formula.
\begin{align*}
 \X =  & \{(Tx, y, X_1, \cdots, X_{r-1}, Y_1, \cdots, Y_{r-1}) \in (T\backslash G)\times G \times \mathfrak{g}^{2r-2} \mid xyx^{-1} \in B, 
\\ &u_j({}^{y^{-1}}X_1, \cdots, {}^{y^{-1}}X_j, Y_1, \cdots, Y_j, X_1, \cdots, X_j) \in {}^{x^{-1}}\mathfrak{b} \textup{ for } 1 \leq j \leq r\}
\end{align*}
Here we use the isomorphism (\ref{giso}) to identify 
\begin{displaymath}
Tg = (Tx)e^{\e X_1} \cdots e^{\e^{r-1} X_{r-1}}, \qquad g' = ye^{\e Y_1} \cdots e^{\e^{r-1} Y_{r-1}}.
\end{displaymath}
Then we have the following diagram
\begin{equation}\label{diag1}
\begin{gathered}
\xymatrix{
&\X \ar[ld]_{\tau} \ar[rd]^{\pi} \ar[d]^{\rho}
\\T_r&\tilde{G}_r \ar[l]_{\tilde{\tau}} \ar[r]^{\tilde{\pi}}&G_r }
\end{gathered}
\end{equation}
where new morphisms are defined as follows.
\begin{align*}
&\rho(Tx, y, X_1, \cdots, X_{r-1}, Y_1, \cdots, Y_{r-1}) = (B_r x e^{\e X_1} \cdots e^{\e^{r-1} X_{r-1}}, y e^{\e Y_1} \cdots e^{\e^{r-1} Y_{r-1}})
\\&\tau(Tx, y, X_1, \cdots, X_{r-1}, Y_1, \cdots, Y_{r-1}) 
\\&= \sigma((x e^{\e X_1} \cdots e^{\e^{r-1} X_{r-1}})(y e^{\e Y_1} \cdots e^{\e^{r-1} Y_{r-1}})(x e^{\e X_1} \cdots e^{\e^{r-1} X_{r-1}})^{-1})
\\&= \sigma(xy e^{\e ({}^{y^{-1}}X_1)} \cdots e^{\e^{r-1} ({}^{y^{-1}}X_{r-1})}e^{\e Y_1} \cdots e^{\e^{r-1} Y_{r-1}}(e^{\e X_1} \cdots e^{\e^{r-1} X_{r-1}})^{-1}x^{-1})
\\&= \sigma({}^x(ye^{\e \U_1} \cdots e^{\e^{r-1} \U_{r-1}}))=\sigma(xyx^{-1} e^{\e ({}^x\U_1)} \cdots e^{\e^{r-1} ({}^x\U_{r-1})})
\\&\pi(Tx, y, X_1, \cdots, X_{r-1}, Y_1, \cdots, Y_{r-1}) = y e^{\e Y_1} \cdots e^{\e^{r-1} Y_{r-1}}
\end{align*}
(Recall that $\sigma: B_r \rightarrow T_r$ is defined above.) Here 
\begin{displaymath}
\U_j = u_j({}^{y^{-1}}X_1, \cdots, {}^{y^{-1}}X_j, Y_1, \cdots, Y_j, X_1, \cdots, X_j)
\end{displaymath}
for $1\leq j \leq r-1$. One can easily check that this definition makes the diagram commute. Therefore, since $\rho$ is an affine space bundle, we have 
\begin{align}\label{comiso}
\pi_! \tau^* (\E \boxtimes \L_{f})&=\tilde{\pi}_!\rho_!\rho^* \tilde{\tau}^* (\E \boxtimes \L_{f})=\tilde{\pi}_! \tilde{\tau}^* (\E \boxtimes \L_{f})[-2d](-d)
\end{align}
where $d=r \dim \b - \dim T$ is the dimension of the fiber of $\X \rightarrow \tilde{G}_r$. Now if we define
\begin{align*}
&\iota : \X \rightarrow T : (Tx, y, X_1, \cdots, X_{r-1}, Y_1, \cdots, Y_{r-1}) \mapsto \overline{xyx^{-1}}\\
&h: \X \rightarrow \k : (Tx, y, X_1, \cdots, X_{r-1}, Y_1, \cdots, Y_{r-1}) \mapsto \sum_{i=1}^{r-1} \langle A_i, {}^x \U_i \rangle = \sum_{i=1}^{r-1} \langle {}^{x^{-1}}A_i, \U_i \rangle
\end{align*}
(where $\U_i$ is the same as above and $\overline{xyx^{-1}}$ is the image of $xyx^{-1} \in B$ under the composition of the quotient morphism $B \rightarrow B/U$ and the inverse of $T \hookrightarrow B \rightarrow B/U$) then by definition we have
\begin{align}\label{comiso2}
\pi_! \tau^* (\E \boxtimes \L_{f}) = \pi_! (\iota^*\E \otimes \L_h).
\end{align}
We wish to compare this to a geometric analogue of the character in \cite{gerardin}. To that end, we first define
\begin{gather*}
G_r^i = \{g \in G_r \mid g \equiv id \mod \e^i\}
\\B_r^i = B_r \cap G_r^i, \quad T_r^i = T_r\cap G_r^i, \quad U_r^i = U_r \cap G_r^i.
\end{gather*}
and define $\Y$ analogous to $\X$ as follows.
\begin{align*}
\Y \colonequals  &\{(Tg, g') \in (T\backslash G_r)\times G_r \mid gg'g^{-1} \in T_{r} \cdot G_{r}^{r/2} \textup{ if } r \textup{ is even},
\\ &\qquad \qquad gg'g^{-1} \in T_r\cdot B_r^{(r-1)/2} \cdot G_r^{(r+1)/2} \textup{ if } r \textup{ is odd}\}
\end{align*}
By the isomorphism (\ref{giso}) it can be also written as
\begin{align*}
\Y =  & \{(Tx, y, X_1, \cdots, X_{r-1}, Y_1, \cdots, Y_{r-1}) \in (T\backslash G)\times G \times \mathfrak{g}^{2r-2} \mid xyx^{-1} \in T, 
\\ &u_j({}^{y^{-1}}X_1, \cdots, {}^{y^{-1}}X_j, Y_1, \cdots, Y_j, X_1, \cdots, X_j) \in {}^{x^{-1}}\mathfrak{t} \textup{ for } 1 \leq j \leq \floor{r/2}-1,
\\ &u_j({}^{y^{-1}}X_1, \cdots, {}^{y^{-1}}X_j, Y_1, \cdots, Y_j, X_1, \cdots, X_j) \in {}^{x^{-1}}\mathfrak{b}  \textup{ for } j=\floor{r/2} \textup{ if } r \textup{ is odd}\}
\end{align*}
Similarly we also define 
\begin{align*}
&\iota': \Y \rightarrow T : (Tx, y, X_1, \cdots, X_{r-1}, Y_1, \cdots, Y_{r-1}) \mapsto xyx^{-1}
\\&h': \Y \rightarrow \k : (Tx, y, X_1, \cdots, X_{r-1}, Y_1, \cdots, Y_{r-1}) \mapsto \sum_{i=1}^{r-1} \langle A_i, {}^x \U_i \rangle = \sum_{i=1}^{r-1} \langle {}^{x^{-1}}A_i, \U_i \rangle
\\&\pi': \Y \rightarrow G_r : (Tx, y, X_1, \cdots, X_{r-1}, Y_1, \cdots, Y_{r-1}) \mapsto y e^{\e Y_1} \cdots e^{\e^{r-1} Y_{r-1}}
\end{align*}
Now we are ready to state the main theorem of this paper.
\begin{theorem}\label{mainthm} If $p \geq r$ and $A_{r-1}$ is regular semisimple, then $\pi_! (\iota^*\E \otimes \L_h) \simeq \pi'_! (\iota'^*\E \otimes \L_{h'})$.
\end{theorem}
In \cite{charring} a strategy is explained to prove that $\pi'_! (\iota'^*\E \otimes \L_{h'})$ is an intersection cohomology complex up to shift. Thus if this theorem is true, then we see that $\tilde{\pi}_! \tilde{\tau}^* (\E \boxtimes \L_{f})$ also has the same property. Another corollary, the equivalence of characters from two different constructions, is explained on Section \ref{equiv}. From now on we focus on proving this theorem.

\section{Strategy to prove the main theorem}
We briefly explain the strategy in \cite[Section 4]{charring} to prove Theorem \ref{mainthm}. First we describe some definitions to be used throughout this paper. (We note that the notations would be slightly different from those in \cite{charring}.) We set $r' = \floor{\frac{r}{2}}, r''=\ceil{\frac{r}{2}}$ such that $r=r'+r''$ and $r'\leq r'' \leq r'+1$. 
In order to compare two complexes from $\X$ and $\Y$, we define the "intermediate steps" between them. For $r''\leq i \leq r$ we define
\begin{align*}
\mathcal{X}_i \colonequals \{(Tg, g') \in (T\backslash G_r) \times G_r \mid gg'g^{-1} \in B_rG_r^{i}\},
\end{align*}
or in other words
\begin{align*}
\mathcal{X}_i \colonequals & \{(Tx, y, X_1, \cdots, X_{r-1}, Y_1, \cdots, Y_{r-1}) \in (T\backslash G)\times G \times \mathfrak{g}^{2r-2} \mid xyx^{-1} \in B, 
\\ &u_j({}^{y^{-1}}X_1, \cdots, {}^{y^{-1}}X_j, Y_1, \cdots, Y_j, X_1, \cdots, X_j) \in {}^{x^{-1}}\mathfrak{b} \textup{ for } 1 \leq j \leq i-1 \}.
\end{align*}
Likewise, for $0\leq i \leq r'$ we define
\begin{align*}
\mathcal{Y}_i \colonequals \{(Tg, g') \in (T\backslash G_r) \times G_r \mid gg'g^{-1} \in T_r B_r^{i}G_r^{r''}\}
\end{align*}
or similarly, for $i=0$ we define
\begin{align*}
\mathcal{Y}_0 \colonequals & \{(Tx, y, X_1, \cdots, X_{r-1}, Y_1, \cdots, Y_{r-1}) \in (T\backslash G)\times G \times \mathfrak{g}^{2r-2} \mid xyx^{-1} \in B, 
\\ &u_j({}^{y^{-1}}X_1, \cdots, {}^{y^{-1}}X_j, Y_1, \cdots, Y_j, X_1, \cdots, X_j) \in {}^{x^{-1}}\mathfrak{b} \textup{ for } 1 \leq j \leq r''-1\}
\end{align*}
and for $1 \leq i \leq r'$ we define
\begin{align*}
\mathcal{Y}_i \colonequals & \{(Tx, y, X_1, \cdots, X_{r-1}, Y_1, \cdots, Y_{r-1}) \in (T\backslash G)\times G \times \mathfrak{g}^{2r-2} \mid xyx^{-1} \in T, 
\\ &u_j({}^{y^{-1}}X_1, \cdots, {}^{y^{-1}}X_j, Y_1, \cdots, Y_j, X_1, \cdots, X_j) \in {}^{x^{-1}}\mathfrak{t} \textup{ for } 1 \leq j \leq i-1,
\\ &u_j({}^{y^{-1}}X_1, \cdots, {}^{y^{-1}}X_j, Y_1, \cdots, Y_j, X_1, \cdots, X_j) \in {}^{x^{-1}}\mathfrak{b} \textup{ for } i \leq j \leq r''-1 \}.
\end{align*}
Note that 
$\X = \X_{r} \subset \cdots \subset \X_{r''} = \Y_{0}  \supset \cdots \supset \Y_{r'} = \Y$. Also we define morphisms $\hat{\pi}, \hat{h}, \hat{\iota}$ on $\X_{r''} = \Y_0$ as follows.
\begin{gather*}
\hat{\pi}(Tx, y, X_1, \cdots, X_{r-1}, Y_1, \cdots, Y_{r-1}) = ye^{\e Y_1}\cdots e^{\e^{r-1} Y_{r-1}},
\end{gather*}
\begin{align*}
&\hat{h}(Tx, y, X_1, \cdots, X_{r-1}, Y_1, \cdots, Y_{r-1}) \\
&= \sum_{j=1}^{r-1} \langle {}^{x^{-1}} A_j, u_j({}^{y^{-1}}X_1, \cdots, {}^{y^{-1}}X_j, Y_1, \cdots, Y_j, X_1, \cdots, X_j)\rangle,
\\&\hat{\iota}(Tx, y, X_1, \cdots, X_{r-1}, Y_1, \cdots, Y_{r-1}) = \overline{xyx^{-1}}.
\end{align*}
(Recall that $\overline{xyx^{-1}} \in T$ is the image of $xyx^{-1} \in B$ under the composition of the quotient morphism $B \rightarrow B/U$ and the inverse of $T \hookrightarrow B \rightarrow B/U$.) Then we have the following commutative diagram.
\begin{displaymath}
\xymatrix{&&G_r&&
\\\X = \X_{r} \ar@{^{(}->}[r] \ar[urr]^{\pi} \ar[d]_(.6){\iota} \ar[drrrr]_(.8){h}
& \cdots \ar@{^{(}->}[r]
& \X_{r''}=\Y_{0} \ar[u]_{\hat{\pi}} \ar[dll]_(.75){\hat{\iota}} \ar[drr]^(.75){\hat{h}}
&\ar@{_{(}->}[l] \cdots
& \ar@{_{(}->}[l] \Y_{r'} = \Y \ar[ull]_{\pi'} \ar[dllll]^(.8){\iota'} \ar[d]^(.6){h'}
\\T&&&&\k\\
}
\end{displaymath}
Now for $r'' \leq i \leq r$, we define 
\begin{displaymath}
L_i = (\hat{\pi}|_{\X_i})_{!}((\hat{\iota}|_{\X_i})^*{\mathcal{E}}\otimes \mathcal{L}_{\hat{h}|_{\X_i}}) \in \D(G_r)
\end{displaymath}
and similarly for $0\leq i \leq r'$ we define 
\begin{displaymath}
L_i' = (\hat{\pi}|_{\Y_i})_{!}((\hat{\iota}|_{\Y_i})^*{\mathcal{E}}\otimes \mathcal{L}_{\hat{h}|_{\Y_i}}) \in \D(G_r).
\end{displaymath}
Then Theorem \ref{mainthm} is equivalent to that $L_{r}=L'_{r'}$. To that end, we will successively show the following.
\begin{displaymath}
L_{r} = L_{r-1} = \cdots = L_{r''}= L'_{0} = \cdots = L'_{r'-1} = L'_{r'}.
\end{displaymath}
(Note that $L_{r''}= L'_{0}$ is automatic.) By \cite[1.4]{bbd}, for $r'' \leq i \leq r-1$ we have a distinguished triangle
\begin{displaymath}
M_i \rightarrow L_i \rightarrow L_{i+1} \xrightarrow{+1}
\end{displaymath}
where 
\begin{displaymath}
M_i \colonequals (\hat{\pi}|_{\X_{i} - \X_{i+1}})_! ((\hat{\iota}|_{\X_{i} - \X_{i+1}})^* \E \otimes \L_{\hat{h}|_{\X_{i} - \X_{i+1}}}).
\end{displaymath}
  Thus if we can show $M_{r''} = \cdots = M_{r-1} = 0$, then the half of the assertion above is justified. Likewise, for the other half it suffices to show $M_0'=\cdots=M_{r'-1}=0$ where
\begin{displaymath}
M_i' \colonequals (\hat{\pi}|_{\Y_{i} - \Y_{i+1}})_! ((\hat{\iota}|_{\Y_{i} - \Y_{i+1}})^* \E \otimes \L_{\hat{h}|_{\Y_{i} - \Y_{i+1}}}).
\end{displaymath}
In the following sections we show that this is indeed the case. First we need some formulae which will be used in the proof.

\section{Some formulae}
Let $r \geq 2$ and suppose we are given the following non-commutative equation.
\begin{equation}\label{firstformula}
e^{\e^i \M_i}\cdots e^{\e^{r-1} \M_{r-1}} e^{\e \UU_1} \cdots e^{\e^{r-1} \UU_{r-1}}e^{-\e^{r-1} \N_{r-1}}\cdots e^{-\e^i \N_i}= e^{\e \VV_1} \cdots e^{\e^{r-1} \VV_{r-1}} \mod \e^r 
\end{equation}
Note that it is indeed a non-commutative polynomial equation. 
\begin{lemma}\label{formula1} Suppose (\ref{firstformula}) is given. For $1 \leq j \leq  r-1$, we have
\begin{align*}
&\VV_j = \UU_j,&j<i\\
&\VV_j = \UU_j + \M_j-\N_j,&j=i\\
&\VV_j = \UU_j +  \M_j-\N_j+\psi_j(\M_i, \cdots, \M_{j-1},\N_i, \cdots, \N_{j-1}, \UU_1, \cdots, \UU_{j-1}),&j >i
\end{align*}
for some Lie polynomial $\psi_j$ which does not depend on $r$ without a constant or linear term. In particular, $\VV_j$ is a Lie polynomial of $\M_i, \cdots, \M_{j},\N_i, \cdots, \N_{j}, \UU_1, \cdots, \UU_{j}$ which does not depend on $r$ whenever $j \leq r-1$. 
\end{lemma}
\begin{proof}
We use induction on $j$. $j\leq i$ case is trivial. Otherwise, from (\ref{firstformula}) we have
\begin{align*}
e^{-\e^{j-1} \VV_{j-1}}\cdots e^{-\e \VV_1} e^{\e^i \M_i}\cdots e^{\e^{r-1} \M_{r-1}} e^{\e \UU_1} \cdots e^{\e^{r-1} \UU_{r-1}}e^{-\e^{r-1} \N_{r-1}}\cdots e^{-\e^i \N_i}
\\= e^{\e^j \VV_j} \cdots e^{\e^{r-1} \VV_{r-1}} \mod \e^r
\end{align*}
We take the logarithm on both sides (it is also a polynomial operation since $\e^{r}=0$) and compare the coefficients of $\e^j$. Then the result follows from (the existence of) the Campbell-Hausdorff formula and induction hypothesis.
\end{proof}
\begin{lemma}\label{formula1.5} Suppose (\ref{firstformula}) is given and keep the notations above. Then we have
\begin{gather*}
\psi_j(t^i\M_i, \cdots, t^{j-1}\M_{j-1},t^i\N_i, \cdots, t^{j-1}\N_{j-1}, t\UU_1, \cdots, t^{j-1}\UU_{j-1}) 
\\= t^j\psi_j(\M_i, \cdots, \M_{j-1},\N_i, \cdots, \N_{j-1}, \UU_1, \cdots, \UU_{j-1})
\end{gather*}
where $t$ is an indeterminate which commutes with all the other variables. In other words, if we decree that the degree of $\M_k, \N_k, \UU_k$ is $k$ for $k \geq 1$, then 
\begin{displaymath}
\psi_j(\M_i, \cdots, \M_{j-1},\N_i, \cdots, \N_{j-1}, \UU_1, \cdots, \UU_{j-1})
\end{displaymath} is a homogeneous polynomial of degree $j$. Furthermore, for $j >i$ we have
\begin{displaymath}
\e^j\VV_j = \e^j\UU_j +  \e^j\M_j-\e^j\N_j+\psi_j(\e^i\M_i, \cdots, \e^{j-1}\M_{j-1},\e^i\N_i, \cdots, \e^{j-1}\N_{j-1}, \e\UU_1, \cdots, \e^{j-1}\UU_{j-1}).
\end{displaymath}
\end{lemma}
\begin{proof} Note that the other assertions clearly follow from the first one. If we replace $\e$ by $\alpha\e$ for some $\alpha \in \mathbb{Q}^\times$, then (\ref{firstformula}) becomes
\begin{align*}
e^{(\alpha\e)^i \M_i}\cdots e^{(\alpha\e)^{r-1} \M_{r-1}} e^{(\alpha\e) \UU_1} \cdots e^{(\alpha\e)^{r-1} \UU_{r-1}}e^{-(\alpha\e)^{r-1} \N_{r-1}}\cdots e^{-(\alpha\e)^i \N_i}
\\= e^{(\alpha\e) \VV_1} \cdots e^{(\alpha\e)^{r-1} \VV_{r-1}} \mod \e^r 
\end{align*}
Therefore we have
\begin{align*}
e^{\e^i (\alpha^i\M_i)}\cdots e^{\e^{r-1} (\alpha^{r-1}\M_{r-1})} e^{\e (\alpha\UU_1)} \cdots e^{\e^{r-1} (\alpha^{r-1}\UU_{r-1})}e^{-\e^{r-1} (\alpha^{r-1}\N_{r-1})}\cdots e^{-\e^i (\alpha^i\N_i)}
\\= e^{\e (\alpha\VV_1)} \cdots e^{\e^{r-1} (\alpha^{r-1}\VV_{r-1})}  \mod \e^r 
\end{align*}
By Lemma \ref{formula1} we have
\begin{align*}
&\alpha^j\VV_j = \alpha^j\UU_j,&j<i\\
&\alpha^j\VV_j = \alpha^j\UU_j + \alpha^j\M_j-\alpha^j\N_j,&j=i\\
&\alpha^j\VV_j = \alpha^j\UU_j + \alpha^j\M_j-\alpha^j\N_j\\
&\qquad\qquad+\psi_j(\alpha^i\M_i, \cdots, \alpha^{j-1}\M_{j-1},\alpha^i\N_i, \cdots, \alpha^{j-1}\N_{j-1}, \alpha\UU_1, \cdots, \alpha^{j-1}\UU_{j-1}),&j >i
\end{align*}
Thus in particular we have
\begin{gather*}
\psi_j(\alpha^i\M_i, \cdots, \alpha^{j-1}\M_{j-1},\alpha^i\N_i, \cdots, \alpha^{j-1}\N_{j-1}, \alpha\UU_1, \cdots, \alpha^{j-1}\UU_{j-1}) 
\\= \alpha^j\psi_j(\M_i, \cdots, \M_{j-1},\N_i, \cdots, \N_{j-1}, \UU_1, \cdots, \UU_{j-1})
\end{gather*}
Since it is true for all $\alpha \in \mathbb{Q}^\times$, the result follows.
\end{proof}

\begin{lemma} \label{formula2} Suppose (\ref{firstformula}) is given and keep the notations above. For $2i<j\leq r-1$, $\VV_j - [\N_i, \UU_{j-i}]$ can be expressed as a polynomial without a term that involves both one of $\UU_{j-i}, \cdots, \UU_{r-1}$ and one of $\M_i, \cdots, \M_{r-1}, \N_i, \cdots, \N_{r-1}$.
\end{lemma}
\begin{proof} From (\ref{firstformula}) we have the following equation.
\begin{align*}
e^{-\e^{j-1} \VV_{j-1}}\cdots e^{-\e \VV_1} e^{\e^i \M_i}\cdots e^{\e^{r-1} \M_{r-1}} e^{\e \UU_1} \cdots e^{\e^{r-1} \UU_{r-1}}e^{-\e^{r-1} \N_{r-1}}\cdots e^{-\e^i \N_i}
\\= e^{\e^j \VV_j} \cdots e^{\e^{r-1} \VV_{r-1}} \mod \e^r
\end{align*}
By setting $r=j+1$, we have
\begin{align*}
e^{-\e^{j-1} \VV_{j-1}}\cdots e^{-\e \VV_1} e^{\e^i \M_i}\cdots e^{\e^{j} \M_{j}} e^{\e \UU_1} \cdots e^{\e^{j} \UU_{j}}e^{-\e^{j} \N_{j}}\cdots e^{-\e^i \N_i}= e^{\e^j \VV_j}\mod \e^{j+i}
\end{align*}
Thus
\begin{equation} \label{instead}
\e^j \VV_j = \log(e^{-\e^{j-1} \VV_{j-1}}\cdots e^{-\e \VV_1} e^{\e^i \M_i}\cdots e^{\e^{j} \M_{j}} e^{\e \UU_1} \cdots e^{\e^{j} \UU_{j}}e^{-\e^{j} \N_{j}}\cdots e^{-\e^i \N_i}) \mod \e^{j+1}.
\end{equation}
Using Lemma \ref{formula1} and \ref{formula1.5} we replace each $\e\VV_1, \cdots, \e^{j-1}\VV_{j-1}$ on the right hand side by 
\begin{align*}
&\e^j\VV_j = \e^j\UU_j,&j<i\\
&\e^j\VV_j = \e^j\UU_j + \e^j\M_j-\e^j\N_j,&j=i\\
&\e^j\VV_j = \e^j\UU_j + \e^j\M_j-\e^j\N_j\\
&\qquad\qquad+\psi_j(\e^i\M_i, \cdots, \e^{j-1}\M_{j-1},\e^i\N_i, \cdots, \e^{j-1}\N_{j-1}, \e\UU_1, \cdots, \e^{j-1}\UU_{j-1}),&j >i
\end{align*}
and collect terms that both one of $\UU_{j-i}, \cdots, \UU_{r-1}$ and one of $\M_i, \cdots, \M_{r-1}, \N_i, \cdots, \N_{r-1}$ appear. Then modulo $\e^{j+1}$ their sum is of the form 
\begin{displaymath}
a(\e^{j-i}\UU_{j-i})(\e^i\M_i) + b(\e^i\M_i)(\e^{j-i}\UU_{j-i}) + c(\e^{j-i}\UU_{j-i})(\e^i\N_i) + d(\e^i\N_i)(\e^{j-i}\UU_{j-i})
\end{displaymath}
or equivalently
\begin{displaymath}
\e^j (a\UU_{j-i}\M_i + b\M_i\UU_{j-i} + c\UU_{j-i}\N_i + d\N_i\UU_{j-i})
\end{displaymath}
for some $a, b, c, d \in \mathbb{Q}$, which we denote by $\tilde{\VV}_j$. Thus $\VV_j - \tilde{\VV}_j$ can be expressed as a polynomial without a term that involves both one of $\UU_{j-i}, \cdots, \UU_{r-1}$ and one of $\M_i, \cdots, \M_{r-1}, \N_i, \cdots, \N_{r-1}$. 

We calculate $\tilde{\VV}_j$. First note that $\psi_{k}(\M_i, \cdots, \M_{k-1},\N_i, \cdots, \N_{k-1}, \UU_1, \cdots, \UU_{k-1})$ is a Lie polynomial without a constant or linear term, and if $k<j$ then it can be expressed without a term that involves both one of $\UU_{j-i}, \cdots, \UU_{r-1}$ and one of $\M_i, \cdots, \M_{r-1}, \N_i, \cdots, \N_{r-1}$. Indeed, this polynomial is homogeneous of degree $k$ with respect to the degree defined in Lemma \ref{formula1.5}, but any term that involves both one of $\UU_{j-i}, \cdots, \UU_{r-1}$ and one of $\M_i, \cdots, \M_{r-1}, \N_i, \cdots, \N_{r-1}$ is of degree $\geq j$. Thus instead of (\ref{instead}) it is equivalent to collect terms that both one of $\UU_{j-i}, \cdots, \UU_{r-1}$ and one of $\M_i, \cdots, \M_{r-1}, \N_i, \cdots, \N_{r-1}$ appear in the following expression given by removing $\psi_{i+1}, \cdots, \psi_{j-1}$.
\begin{equation} \label{instead2}
\begin{aligned}
\log(e^{-\e^{j-1} (\UU_{j-1}+\M_{j-1}-\N_{j-1})}\cdots e^{-\e^{i} (\UU_{i}+\M_{i}-\N_{i})}e^{-\e^{i-1} \UU_{i-1}}\cdots e^{-\e \UU_1} 
\\e^{\e^i \M_i}\cdots e^{\e^{j} \M_{j}} e^{\e \UU_1} \cdots e^{\e^{j} \UU_{j}}e^{-\e^{j} \N_{j}}\cdots e^{-\e^i \N_i}) \mod \e^{j+1}.
\end{aligned}
\end{equation}

Also, as $\tilde{\VV}_j$ can be expressed with indeterminates $\UU_{j-i}, \M_i, \N_i$, the values $a, b, c, d$ are unchanged if we simply put $\UU_k=0$ for $k\neq j-i$ and $\M_{i+1} = \cdots = \M_{r-1} = \N_{i+1}=\cdots = \N_{r-1} = 0$. It means instead of (\ref{instead2}) it suffices to collect terms which both one of $\UU_{j-i}, \cdots, \UU_{r-1}$ and one of $\M_i, \cdots, \M_{r-1}, \N_i, \cdots, \N_{r-1}$ appear in the following expression.
$$\log(e^{-\e^{j-i}\UU_{j-i}}e^{-\e^{i}(\M_i-\N_i)}e^{\e^i \M_i} e^{\e^{j-i} \UU_{j-i}}e^{-\e^i \N_i}) \mod \e^{j+1}.$$
(Note that $j-i>i$ by assumption.) By the Campbell-Hausdorff formula, we see that $\tilde{\VV}_j = [\N, \UU_{j-i}]$ satisfies the desired property.
\end{proof}

\begin{lemma}\label{formula1.75} Suppose (\ref{firstformula}) is given and keep the notations above. Then for $j > i$, we have
\begin{displaymath}\psi_j(0, \cdots, 0,0, \cdots, 0, \UU_1, \cdots, \UU_{j-1})=0.
\end{displaymath}
\end{lemma}
\begin{proof} It is straightforward if we replace $\M_i, \cdots, \M_{r-1}, \N_i, \cdots, \N_{r-1}$ by $0$ on (\ref{firstformula}).
\end{proof}

\begin{lemma}\label{formula1.875} Suppose (\ref{firstformula}) is given and keep the notations above. Then for $j > i$, 
\begin{displaymath}
\psi_j(\M_i, \cdots, \M_{j-1},\N_i, \cdots, \N_{j-1}, \UU_1, \cdots, \UU_{j-1})
\end{displaymath}
can be expressed as a Lie polynomial each of whose terms contains at least one of $\M_i, \cdots, \M_{j-1},$ $\N_i, \cdots, \N_{j-1}$.
\end{lemma}
\begin{proof} It is a natural consequence of Lemma \ref{formula1.75}. Suppose we have an expression $A+B$ of $\psi_j(\M_i, \cdots, \M_{j-1},\N_i, \cdots, \N_{j-1}, \UU_1, \cdots, \UU_{j-1})$ where $A$ and $B$ are expressions of Lie polynomials, all the terms of $A$ contain at least one of $\M_i, \cdots, \M_{j-1},\N_i, \cdots, \N_{j-1}$, and $B$ consists only of terms in variables $\UU_1, \cdots, \UU_{j-1}$. Then by Lemma \ref{formula1.75}, we have $B=0$, thus the result follows.
\end{proof}

From now on, instead of (\ref{firstformula}) we assume the following equation is given.
\begin{equation}\label{secondformula}
e^{\e^i \M}e^{\e \UU_1} \cdots e^{\e^{r-1} \UU_{r-1}}e^{-\e^i \N}= e^{\e \VV_1} \cdots e^{\e^{r-1} \VV_{r-1}} \mod \e^r 
\end{equation}

\begin{lemma}\label{formula3} Suppose (\ref{secondformula}) is given. For $i<j<2i$, we have 
\begin{displaymath}
\VV_j = \UU_j + \sum_{\lambda \vdash j-i} \frac{(-1)^{\ell(\lambda)}}{\prod_{i \geq 1} \lambda(i)!}\ola{ad}^\lambda_{\UU} (\M)\end{displaymath}
where the sum is over all partitions of $j-i$ and
\begin{itemize}
\item[-] $\ell(\lambda) = \lambda'_1$ denotes the length of $\lambda$,
\item[-] $\lambda(i)$ denotes the number of $i$ in $\lambda$,
\item[-] $\ola{ad}^\lambda_\UU \colonequals \cdots ad_{\UU_i}^{\lambda(i)}\cdots ad_{\UU_1}^{\lambda(1)}$.
\end{itemize}
\end{lemma}
\begin{proof} Here we set $r=2i$. Then we have
\begin{align*}
&e^{-\e^{i-1} \VV_{i-1}}\cdots e^{-\e \VV_1} e^{\e^i \M} e^{\e \UU_1} \cdots e^{\e^{i-1} \UU_{i-1}}e^{-\e^i \N} = e^{\e^{i} (\VV_{i}-\UU_{i})} \cdots e^{\e^{2i-1} (\VV_{2i-1}-\UU_{2i-1})} \mod \e^{2i}.
\end{align*}
Here we use the fact that $e^{\e^s X}$ and $e^{\e^t Y}$ commute for $s+t\geq 2i$. By Lemma \ref{formula1}, it reads
\begin{align*}
&e^{-\e^{i-1} \UU_{i-1}}\cdots e^{-\e \UU_1} e^{\e^i \M} e^{\e \UU_1} \cdots e^{\e^{i-1} \UU_{i-1}} = e^{\e^{i} \M} e^{\e^{i+1} (\VV_{i+1}-\UU_{i+1})} \cdots e^{\e^{2i-1} (\VV_{2i-1}-\UU_{2i-1})} \mod \e^{2i}.
\end{align*}
Now the result follows from comparing coefficients of $\e^j$ on each side of the equation above and the following formula.
\begin{displaymath}
e^{X}Ye^{-X} = \sum_{i\geq 0} \frac{1}{i!}ad_X^i(Y)
\end{displaymath}
\end{proof}

\section{Vanishing of $M_i$} 
It is proved in \cite[4.4]{charring} that $M_{r-1}=0$. Thus here we assume that $r'' \leq i \leq r-2$ and prove $M_{i}=0$. We have
\begin{align*}
\mathcal{X}_i - \mathcal{X}_{i+1} = & \{(Tx, y, X_1, \cdots, X_{r-1}, Y_1, \cdots, Y_{r-1}) \in (T\backslash G)\times G \times \mathfrak{g}^{2r-2} \mid xyx^{-1} \in B, 
\\ &u_j({}^{y^{-1}}X_1, \cdots, {}^{y^{-1}}X_j, Y_1, \cdots, Y_j, X_1, \cdots, X_j) \in {}^{x^{-1}}\mathfrak{b} \textup{ for } 1 \leq j \leq i-1,
\\ &u_j({}^{y^{-1}}X_1, \cdots, {}^{y^{-1}}X_j, Y_1, \cdots, Y_j, X_1, \cdots, X_j) \notin {}^{x^{-1}}\mathfrak{b}  \textup{ for } j=i\}.
\end{align*}
Then the projection $\hat{\pi}: \mathcal{X}_i - \mathcal{X}_{i+1} \rightarrow G_r$ factors through
$\bar{\pi}: \mathcal{X}_i - \mathcal{X}_{i+1}\rightarrow (T\backslash G)\times G_r \rightarrow G_r$ where\begin{displaymath}
\bar{\pi}: (Tx, y, X_1, \cdots, X_{r-1}, Y_1, \cdots, Y_{r-1}) \mapsto (Tx, ye^{\e Y_1}\cdots e^{\e^{r-1} Y_{r-1}})
\end{displaymath}
and $(T\backslash G)\times G_r \rightarrow G_r$ is the second projection. Thus in order to prove $M_i=0$, it suffices to show that $\bar{\pi}_!((\hat{\iota}|_{\X_{i} - \X_{i+1}})^* \E \otimes \L_{\hat{h}|_{\X_{i} - \X_{i+1}}})=0$. Now we fix a fiber $\P = \bar{\pi}^{-1}(Tx, ye^{\e Y_1}\cdots e^{\e^{r-1} Y_{r-1}})$ for some fixed $Tx, y, Y_1, \cdots, Y_{r-1}$. Then it suffices to show that $(\bar{\pi}|_{\P})_!((\hat{\iota}|_{\P})^* \E \otimes \L_{\hat{h}|_{\P}})=0$, i.e.
\begin{displaymath} H^*_c(\P, (\hat{\iota}|_{\P})^* \E \otimes \L_{\hat{h}|_{\P}})=0.
\end{displaymath}
But since we fix $Tx$ and $y$, the pull-back of $\E$ under $\hat{\iota}$ is a constant sheaf on $\P$, thus it suffices to show that
\begin{displaymath} H^*_c(\P,  \L_{\hat{h}|_{\P}})=0.
\end{displaymath}
From now on we identify
\begin{align*}
\mathcal{P} = &\{(X_1, \cdots, X_{r-1}) \in \mathfrak{g}^{r-1} \mid
\\ &u_j({}^{y^{-1}}X_1, \cdots, {}^{y^{-1}}X_j, Y_1, \cdots, Y_j, X_1, \cdots, X_j) \in {}^{x^{-1}}\mathfrak{b} \textup{ for } 1 \leq j \leq i-1,
\\ &u_j({}^{y^{-1}}X_1, \cdots, {}^{y^{-1}}X_j, Y_1, \cdots, Y_j, X_1, \cdots, X_j) \notin {}^{x^{-1}}\mathfrak{b}  \textup{ for } j=i\}
\end{align*}
and fix a representative $x \in Tx$. Also, note that we have an isomorphism of varieties $\n^{i+1} \simeq {}^{x^{-1}}U_r^{r-1-i}$ given by 
\begin{equation}\label{niso}
(E_{r-1-i}, \cdots, E_{r-1}) \mapsto e^{\e^{r-1-i}({}^{x^{-1}}E_{r-1-i})}\cdots e^{\e^{r-1}({}^{x^{-1}}E_{r-1})}.\end{equation}
We define a free action of ${}^{x^{-1}}U_r^{r-1-i}$ on $\P$,
\begin{displaymath}(e^{\e^{r-1-i}({}^{x^{-1}}E_{r-1-i})}\cdots e^{\e^{r-1}({}^{x^{-1}}E_{r-1})})\cdot (X_1, \cdots, X_{r-1}) = (X'_1, \cdots, X'_{r-1})
\end{displaymath}
such that the following identity holds.
\begin{displaymath}(e^{\e^{r-1-i}({}^{x^{-1}}E_{r-1-i})}\cdots e^{\e^{r-1}({}^{x^{-1}}E_{r-1})})(e^{\e X_1} \cdots e^{\e^{r-1} X_{r-1}}) = e^{\e X'_1} \cdots e^{\e^{r-1} X'_{r-1}}\end{displaymath}
(In other words, it is just a "left multiplication".) At least it is clear that $X_j'$ are determined uniquely. We have
\begin{align}
&e^{\e ({}^{y^{-1}}X'_1)} \cdots e^{\e^{r-1} ({}^{y^{-1}}X'_{r-1})}e^{\e Y_1} \cdots e^{\e^{r-1} Y_{r-1}}(e^{\e X'_1} \cdots e^{\e^{r-1} X'_{r-1}})^{-1} \label{exp1}
\\&=(e^{\e^{r-1-i}({}^{(xy)^{-1}}E_{r-1-i})}\cdots e^{\e^{r-1}({}^{(xy)^{-1}}E_{r-1})})e^{\e \U_1} \cdots e^{\e^{r-1} \U_{r-1}}(e^{\e^{r-1-i}({}^{x^{-1}}E_{r-1-i})}\cdots e^{\e^{r-1}({}^{x^{-1}}E_{r-1})})^{-1} \nonumber
\end{align}
where $\U_j = u_j({}^{y^{-1}}X_1, \cdots, {}^{y^{-1}}X_j, Y_1, \cdots, Y_j, X_1, \cdots, X_j)$. If we define 
\begin{displaymath}\V_j = u_j({}^{y^{-1}}X'_1, \cdots, {}^{y^{-1}}X'_j, Y_1, \cdots, Y_j, X'_1, \cdots, X'_j)\end{displaymath}
then the expression (\ref{exp1}) is the same as $e^{\e \V_1} \cdots e^{\e^{r-1} \V_{r-1}}$.

The action is indeed well-defined; we need to check that $e^{\e \V_1} \cdots e^{\e^{r-1} \V_{r-1}} \in {}^{x^{-1}}B_rG_r^i - {}^{x^{-1}}B_rG_r^{i+1}.$ But it is clear since 
\begin{displaymath}
e^{\e^{r-1-i}({}^{x^{-1}}E_{r-1-i})}\cdots e^{\e^{r-1}({}^{x^{-1}}E_{r-1})}, \quad e^{\e^{r-1-i}({}^{(xy)^{-1}}E_{r-1-i})}\cdots e^{\e^{r-1}({}^{(xy)^{-1}}E_{r-1})}\end{displaymath} are in ${}^{x^{-1}}B_r$. (Here we use the fact that $xyx^{-1} \in T$ normalizes $\n$.)

Now we fix an ${}^{x^{-1}}U_r^{r-1-i}$-orbit of $(Tx, y, X_1, \cdots, X_{r-1}, Y_1, \cdots, Y_{r-1}) \in \mathcal{P}$, denoted by $\OO$. Then it suffices to show that
\begin{displaymath} H^*_c(\OO,  \L_{\hat{h}|_{\OO}})=0.
\end{displaymath}
But since the action is free, we may identify ${}^{x^{-1}}N_r^{r-1-i} \simeq \OO$ and furthermore $\n^{i+1} \simeq \OO$ by the isomorphism (\ref{niso}). Thus we may regard $\hat{h}$ as a function from $\n^{i+1}$ to $\k$, say,
\begin{displaymath}\hat{h}(E_{r-1-i}, \cdots, E_{r-1}) = \sum_{j=1}^{r-1} \langle {}^{x^{-1}}A_j, \V_j(E_{r-1-i}, \cdots, E_{r-1}) \rangle\end{displaymath}
where we also regard $\V_j$ as a function on $\n^{i+1}$ with $\U_j$ fixed. Now we need the following lemma. The proof below is due to G. Lusztig.
\begin{lemma} If $f : \mathbb{A}^N \rightarrow \k$ is a non-constant affine linear function, then
\begin{displaymath}
H_c^*(\mathbb{A}^N, \L_f) = 0.\end{displaymath}
\end{lemma}
\begin{proof} We recall the construction of $\L_f$ in \cite[0.3]{charring}. We define
\begin{displaymath}
\mathbb{A}^N_f = \{(x, \lambda) \in \mathbb{A}^N \times \k \mid \lambda^q - \lambda = f(x)\}
\end{displaymath}
and let $a : \mathbb{A}^N_f \rightarrow \mathbb{A}^N$ be the projection on the first factor. Then $\L_f$ is defined as the $\psi$-eigenspace of $a_!\overline{\mathbb{Q}_\ell}$. Therefore, to this end it suffices to show that $H^*_c(\mathbb{A}^N, a_!\overline{\mathbb{Q}_\ell})$ has no such eigenspace, i.e. $\F_q$ acts trivially on $H^*_c(\mathbb{A}^N_f, \overline{\mathbb{Q}_\ell})$. 

We may write $f=\sum_{i=1}^n a_i x_i + c$ for some $a_i \in \k$ and without loss of generality we assume $a_1 \neq 0$. Then 
\begin{displaymath}
\mathbb{A}^N_f \rightarrow \mathbb{A}^{N-1} \times \k : (x_1, \cdots, x_n, \lambda) \mapsto (x_2, \cdots, x_n, \lambda)
\end{displaymath}
is an $\F_q$-equivariant isomorphism with inverse
\begin{displaymath}
 (x_2, \cdots, x_n, \lambda) \mapsto (-\frac{1}{a_1}(\sum_{i=2}^n a_i x_i + c - \lambda^q + \lambda), x_2, \cdots, x_n, \lambda)
\end{displaymath}
if we define an action of $x\in \F_q$ on $\mathbb{A}^{N-1}\times\k$ by the following.
$$x : (y_1, \cdots, y_{n-1}, \lambda) \mapsto (y_1, \cdots, y_{n-1}, \lambda+x)$$
However it is clear that this $\F_q$-action on $\mathbb{A}^{N-1} \times \k$ is naturally extended to that of $\k$. Since $\k$ is connected, it acts trivially on the cohomology and the result follows.
\end{proof}
Thus in our case it suffices to show that $\hat{h}(E_{r-1-i}, \cdots, E_{r-1})$ is a non-constant affine linear function on $\n^{i+1}$. We use Lemma \ref{formula1} to see that
\begin{align*}
&\sum_{j=1}^{r-1} \langle {}^{x^{-1}}A_j, \V_j(E_{r-1-i}, \cdots, E_{r-1})  \rangle 
\\&=\sum_{j=1}^{r-1} \langle {}^{x^{-1}}A_j, \U_j \rangle+ \sum_{j=r-1-i}^{r-1}\langle  {}^{x^{-1}}A_{j}, {}^{(xy)^{-1}}E_j- {}^{x^{-1}}E_j \rangle 
\\&+ \sum_{j=r-i}^{r-1} \langle  {}^{x^{-1}}A_j, \psi_j ({}^{(xy)^{-1}}E_{r-1-i}, \cdots, {}^{(xy)^{-1}}E_{j-1},{}^{x^{-1}}E_{r-1-i},\cdots, {}^{x^{-1}}E_{j-1}, \U_1, \cdots, \U_{j-1})\rangle
\end{align*}
where $\psi_j$ for $r-i \leq j \leq r-1$ is a Lie polynomial without a constant or linear term. But $\langle  {}^{x^{-1}}A_{j}, {}^{(xy)^{-1}}E_j- {}^{x^{-1}}E_j \rangle=0$ since ${}^{(xy)^{-1}}E_j- {}^{x^{-1}}E_j \in {}^{x^{-1}}\n$. Thus we have 
\begin{align*}
&\hat{h}(E_{r-1-i}, \cdots, E_{r-1}) =\sum_{j=1}^{r-1} \langle {}^{x^{-1}}A_j, \U_j \rangle
\\&+ \sum_{j=r-i}^{r-1} \langle  {}^{x^{-1}}A_j, \psi_j ({}^{(xy)^{-1}}E_{r-1-i}, \cdots, {}^{(xy)^{-1}}E_{j-1},{}^{x^{-1}}E_{r-1-i},\cdots, {}^{x^{-1}}E_{j-1}, \U_1, \cdots, \U_{j-1})\rangle.
\end{align*}

Note that $\U_1, \cdots, \U_{i-1}\in {}^{x^{-1}}\b$ normalize ${}^{x^{-1}}\n$. Thus by Lemma \ref{formula1.5}  and \ref{formula1.875}, for $j < r-1$ we have
\begin{displaymath}
\langle  {}^{x^{-1}}A_j, \psi_j ({}^{(xy)^{-1}}E_{r-1-i}, \cdots, {}^{(xy)^{-1}}E_{j-1},{}^{x^{-1}}E_{r-1-i},\cdots, {}^{x^{-1}}E_{j-1}, \U_1, \cdots, \U_{j-1})\rangle =0.
\end{displaymath}
Indeed, as shown in the proof of Lemma \ref{formula2}, $\psi_j$ can be expressed without a term that involves both one of $\U_{i}, \cdots, \U_{r-1}$ and one of ${}^{(xy)^{-1}}E_{r-1-i}, \cdots, {}^{(xy)^{-1}}E_{r-1},{}^{x^{-1}}E_{r-1-i},\cdots, {}^{x^{-1}}E_{r-1}$.
Also if $j = r-1$ we have
\begin{align*}
&\langle  {}^{x^{-1}}A_{r-1}, \psi_{r-1} ({}^{(xy)^{-1}}E_{r-1-i}, \cdots, {}^{(xy)^{-1}}E_{r-2},{}^{x^{-1}}E_{r-1-i},\cdots, {}^{x^{-1}}E_{r-2}, \U_1, \cdots, \U_{r-2})\rangle
\\&=\langle  {}^{x^{-1}}A_{r-1}, \psi_{r-1} ({}^{(xy)^{-1}}E_{r-1-i},0, \cdots,0,{}^{x^{-1}}E_{r-1-i},0,\cdots, 0, 0, \cdots,0, \U_i, 0, \cdots, 0)\rangle.
\end{align*}
Now we use Lemma \ref{formula2} to see that it is the same as
\begin{displaymath} \langle {}^{x^{-1}}A_{r-1},[{}^{x^{-1}}E_{r-1-i}, \U_{i}]\rangle\end{displaymath}
since $r-1 > 2(r-1-i)$. Thus
\begin{align*}
\hat{h}(E_{r-1-i}, \cdots, E_{r-1}) &= \sum_{j=1}^{r-1} \langle {}^{x^{-1}}A_j, \U_j \rangle + \langle  {}^{x^{-1}}A_{r-1}, [{}^{x^{-1}}E_{r-1-i}, \U_{i}] \rangle
\\&=\sum_{j=1}^{r-1} \langle {}^{x^{-1}}A_j, \U_j \rangle + \langle  [{}^{x^{-1}}A_{r-1}, {}^{x^{-1}}E_{r-1-i}], \U_{i} \rangle.
\end{align*} 
Note that 
\begin{gather*}
E \mapsto  [A_{r-1}, E] \mapsto [{}^{x^{-1}}A_{r-1}, {}^{x^{-1}}E]
\end{gather*}
is a vector space isomorphism from $\n$ to ${}^{x^{-1}}\n$. Since $\U_i \notin {}^{x^{-1}}\b$, we conclude that $\hat{h}$ is a non-constant affine linear function. Thus $M_i$ vanishes as desired.

\section{Vanishing of $M_i'$}
Since $M_0'=0$ is already shown in \cite[4.2]{charring}, here we assume $1 \leq i \leq r'-1$ and show $M_i'=0$. (This condition is vacuous if $r \leq 3$.) We have
\begin{align*}
\Y_i- \Y_{i+1} = & \{(Tx, y, X_1, \cdots, X_{r-1}, Y_1, \cdots, Y_{r-1}) \in (T\backslash G)\times G \times \mathfrak{g}^{2r-2} \mid xyx^{-1} \in T, 
\\ &u_j({}^{y^{-1}}X_1, \cdots, {}^{y^{-1}}X_j, Y_1, \cdots, Y_j, X_1, \cdots, X_j) \in {}^{x^{-1}}\mathfrak{t} \textup{ for } 1 \leq j \leq i-1,
\\ &u_j({}^{y^{-1}}X_1, \cdots, {}^{y^{-1}}X_j, Y_1, \cdots, Y_j, X_1, \cdots, X_j) \in {}^{x^{-1}}\mathfrak{b} - {}^{x^{-1}}\mathfrak{t}  \textup{ for } j=i,
\\ &u_j({}^{y^{-1}}X_1, \cdots, {}^{y^{-1}}X_j, Y_1, \cdots, Y_j, X_1, \cdots, X_j) \in {}^{x^{-1}}\mathfrak{b} \textup{ for } i+1 \leq j \leq r''-1 \}.
\end{align*}
We define a free action of $E\in \g$ on $\Y_i - \Y_{i+1}$,
\begin{displaymath}
E \cdot (Tx, y, X_1, \cdots, X_{r-1}, Y_1, \cdots, Y_{r-1}) = (Tx, y, X'_1, \cdots, X_{r-1}', Y_1, \cdots, Y_{r-1}),
\end{displaymath}
such that it satisfies the following equation.
\begin{displaymath}
e^{\e^{r-1-i}E}e^{\e X_1} \cdots e^{\e^{r-1} X_{r-1}} = e^{\e X'_1} \cdots e^{\e^{r-1} X'_{r-1}} \mod \e^r
\end{displaymath}
Note that $X_1', \cdots, X_{r-1}'$ is successively well-defined. Then we have
\begin{align}
&e^{\e ({}^{y^{-1}}X'_1)} \cdots e^{\e^{r-1} ({}^{y^{-1}}X'_{r-1})}e^{\e Y_1} \cdots e^{\e^{r-1} Y_{r-1}}(e^{\e X'_1} \cdots e^{\e^{r-1} X'_{r-1}})^{-1}\label{exp2}
\\&=e^{\e^{r-1-i}({}^{y^{-1}}E)} e^{\e \U_1} \cdots e^{\e^{r-1} \U_{r-1}} e^{-\e^{r-1-i}E} \nonumber
\end{align}
where $\U_j = u_j({}^{y^{-1}}X_1, \cdots, {}^{y^{-1}}X_j, Y_1, \cdots, Y_j, X_1, \cdots, X_j)$. If we define 
\begin{displaymath}
\V_j = u_j({}^{y^{-1}}X'_1, \cdots, {}^{y^{-1}}X'_j, Y_1, \cdots, Y_j, X'_1, \cdots, X'_j)\end{displaymath}
then the expression (\ref{exp2}) is the same as $e^{\e \V_1} \cdots e^{\e^{r-1} \V_{r-1}}$. Since $X_j = X_j'$ for $1 \leq j \leq r-2-i$ and especially for $1 \leq j \leq r''-1$, we also have $\U_j = \V_j$ in this range. It implies that the action is well-defined on $\Y_i - \Y_{i+1}$.

By the same reason as the previous chapter, for the proof of vanishing of $M'_i$ it suffices to show that $\hat{h}$ is a non-constant affine linear function on any orbit of $\g$ on $\Y_i - \Y_{i+1}$ under this action. In other words, if we regard $\hat{h}$ as a function of $E \in \g$ by the following formula,
\begin{displaymath}
\hat{h}(E) = \sum_{j=1}^{r-1} \langle {}^{x^{-1}}A_j, \V_j(E)\rangle
\end{displaymath}
it suffices to show that this is a non-constant affine linear function. Now we use Lemma \ref{formula3} to see that
\begin{align*}
\sum_{j=1}^{r-1} \langle {}^{x^{-1}}A_j, \V_j(E)\rangle &= \sum_{j=1}^{r-1} \langle {}^{x^{-1}}A_j, \U_j\rangle + \langle {}^{x^{-1}}A_{r-1-i}, {}^{y^{-1}}E - E \rangle 
\\&\quad + \sum_{j=r-i}^{r-1} \langle {}^{x^{-1}}A_{j}, \sum_{\lambda \vdash j-(r-1-i)} \frac{(-1)^{\ell(\lambda)}}{\prod_{i \geq 1} \lambda(i)!}\ola{ad}_{\U}^\lambda({}^{y^{-1}}E) \rangle
\\&= \sum_{j=1}^{r-1} \langle {}^{x^{-1}}A_j, \U_j\rangle - \langle {}^{x^{-1}}A_{r-1}, [\U_{i}, {}^{y^{-1}}E] \rangle
\\&= \sum_{j=1}^{r-1} \langle {}^{x^{-1}}A_j, \U_j\rangle - \langle [{}^{x^{-1}}A_{r-1}, \U_{i}], {}^{y^{-1}}E \rangle
\end{align*}
since $\langle {}^{x^{-1}}A_{r-1-i}, {}^{y^{-1}}E - E \rangle=0$ and $\U_j \in {}^{x^{-1}}\mathfrak{t}$ for $1 \leq j \leq i-1$. Now since $\U_{i} \notin {}^{x^{-1}}\mathfrak{t}$, $[{}^{x^{-1}}A_{r-1}, \U_{i}] \neq 0$, thus $ \langle [{}^{x^{-1}}A_{r-1}, \U_{i}], {}^{y^{-1}}E \rangle$ is a non-constant affine linear function. Thus we conclude that so is $\hat{h}(E)$ and the claim is proved.

\section{Equality of two characters}\label{equiv}
Theorem \ref{mainthm} has the following consequence which we explain in this section. We recall the diagram (\ref{diag1}).
\begin{displaymath}
\xymatrix{
&\X \ar[ld]_{\tau} \ar[rd]^{\pi} \ar[d]^{\rho}
\\T_r&\tilde{G}_r \ar[l]_{\tilde{\tau}} \ar[r]^{\tilde{\pi}}&G_r }
\end{displaymath}
We saw in (\ref{comiso}) and (\ref{comiso2}) that  $\pi_! (\iota^*\E \otimes \L_h)
 = \pi_! \tau^* (\E \boxtimes \L_{f})=\tilde{\pi}_! \tilde{\tau}^* (\E \boxtimes \L_{f})[-2d](-d)$ where $d=r \dim \b - \dim T$.
 
$\pi'_! (\iota'^*\E \otimes \L_{h'})$ also has a similar picture. We recall the definition of $\Y$
\begin{align*}
\Y \colonequals  &\{(Tg, g') \in (T\backslash G_r)\times G_r \mid gg'g^{-1} \in T_r B_r^{r'}G_r^{r''}\}.
\end{align*}
Now we define
$$
\tilde{G}_r' \colonequals \{(T_r B_r^{r'}G_r^{r''} g, g') \in (T_r B_r^{r'}G_r^{r''}\backslash G_r) \times G_r \mid gg'g^{-1} \in T_r B_r^{r'}G_r^{r''}\}.
$$
Then we have the following diagram
\begin{displaymath}
\xymatrix{
T_r&\tilde{G}_r' \ar[l]_{\tilde{\tau}'} \ar[r]^{\tilde{\pi}'}&G_r }
\end{displaymath}
where
\begin{gather*}
\tilde{\pi}'(T_r B_r^{r'}G_r^{r''}g, g') = g',\qquad \tilde{\tau}'(T_r B_r^{r'}G_r^{r''}g, g') = \sigma'(gg'g^{-1}).
\end{gather*}
Here $\sigma': T_rB_r^{r'}G_r^{r''} \rightarrow T_r$ is defined as follows. Under the isomorphism (\ref{giso}) we define
\begin{align*}
K = \{& (x, X_1, \cdots, X_{r-1}) \in G \times \g^{r-1} \mid
\\& x=1, X_1=\cdots= X_{r'-1}=0,  X_{r''}, \cdots, X_{r-1} \in \n \oplus \n^-, X_{r'} \in \n \textup{ if } r \textup{ is odd} \}
\end{align*}
which is a subgroup of $T_rB_r^{r'}G_r^{r''}$. ($\n^-$ is the nilpotent radical of the Borel subalgebra opposite to $\b$.) An easy calculation shows that this is a normal subgroup of $T_rB_r^{r'}G_r^{r''}$ and $T_r \hookrightarrow T_rB_r^{r'}G_r^{r''} \rightarrow T_rB_r^{r'}G_r^{r''}/K$ gives an isomorphism of algebraic groups. We define $\sigma'$ to be the composition of the quotient morphism by $K$ with the inverse of the isomorphism $T_r \simeq T_rB_r^{r'}G_r^{r''}/K$ above.

Now it is possible to define the diagram analogous to (\ref{diag1}),
\begin{displaymath}
\xymatrix{
&\Y \ar[ld]_{\tau'} \ar[rd]^{\pi'} \ar[d]^{\rho'}
\\T_r&\tilde{G}'_r \ar[l]_{\tilde{\tau}'} \ar[r]^{\tilde{\pi}'}&G_r }
\end{displaymath}
where new morphisms are defined as follows.
\begin{align*}
&\rho'(Tx, y, X_1, \cdots, X_{r-1}, Y_1, \cdots, Y_{r-1}) = (T_r B_r^{r'}G_r^{r''} x e^{\e X_1} \cdots e^{\e^{r-1} X_{r-1}}, y e^{\e Y_1} \cdots e^{\e^{r-1} Y_{r-1}})
\\&\tau'(Tx, y, X_1, \cdots, X_{r-1}, Y_1, \cdots, Y_{r-1}) 
\\&=  \sigma'((x e^{\e X_1} \cdots e^{\e^{r-1} X_{r-1}})(y e^{\e Y_1} \cdots e^{\e^{r-1} Y_{r-1}})(x e^{\e X_1} \cdots e^{\e^{r-1} X_{r-1}})^{-1})
\\&= \sigma'(xy e^{\e ({}^{y^{-1}}X_1)} \cdots e^{\e^{r-1} ({}^{y^{-1}}X_{r-1})}e^{\e Y_1} \cdots e^{\e^{r-1} Y_{r-1}}(e^{\e X_1} \cdots e^{\e^{r-1} X_{r-1}})^{-1}x^{-1})
\\&=\sigma'({}^x(ye^{\e \U_1} \cdots e^{\e^{r-1} \U_{r-1}}))=\sigma'(xyx^{-1} e^{\e ({}^x\U_1)} \cdots e^{\e^{r-1} ({}^x\U_{r-1})})
\\&\pi'(Tx, y, X_1, \cdots, X_{r-1}, Y_1, \cdots, Y_{r-1}) = y e^{\e Y_1} \cdots e^{\e^{r-1} Y_{r-1}}
\end{align*}
By the same reason as (\ref{comiso}) and (\ref{comiso2}) we have
$$\pi'_! (\iota'^*\E \otimes \L_{h'})
 = \pi'_! \tau'^* (\E \boxtimes \L_{f})=\tilde{\pi}'_! \tilde{\tau}'^* (\E \boxtimes \L_{f})[-2d](-d).$$
Therefore we have
\begin{cor}\label{maincor} If $p \geq r$ and $A_{r-1}$ is regular semisimple, then $\tilde{\pi}_! \tilde{\tau}^* (\E \boxtimes \L_{f}) \simeq \tilde{\pi}'_! \tilde{\tau}'^* (\E \boxtimes \L_{f})$.
\end{cor}

Now we suppose that $G$, $B$, and $T$ are defined over $\F_q$ and that $\E$ is a character sheaf which corresponds to some character $\chi : T(\F_q) \rightarrow \mathbb{C}$. Also we suppose that $f$ is defined over $\F_q$ so that $\psi \circ f$ defines a character of $\t(\F_q)^{r-1}$. (Recall that we fixed an additive character $\psi: \F_q \rightarrow \mathbb{C}^*$ before.) We may regard the character $\chi \times (\psi\circ f) : T(\F_q)\times \t(\F_q)^{r-1} \rightarrow \mathbb{C}$ as that of $T_r(\F_q)$ by the isomorphism (\ref{tiso}).

Then by construction the trace of $\tilde{\pi}_! \tilde{\tau}^* (\E \boxtimes \L_{f})$ corresponds to the character 
\begin{displaymath}\Ind^{G_r(\F_q)}_{B_r(\F_q)} (\chi \times (\psi\circ f)) : G_r(\F_q) \rightarrow \mathbb{C}.
\end{displaymath}
Here $\chi \times (\psi\circ f)$ is regarded as a character on ${B_r(\F_q)}$ by pullback under $\sigma : B_r(\F_q) \rightarrow T_r(\F_q)$.

Similarly, $\chi \times (\psi\circ f)$ can also be regarded as a character on $T_r(\F_q)B_r^{r'}(\F_q)G_r^{r''}(\F_q)$ by pullback under $\sigma' : T_r(\F_q)B_r^{r'}(\F_q)G_r^{r''}(\F_q) \rightarrow T_r(\F_q)$. Thus by construction $\tilde{\pi}'_! \tilde{\tau}'^* (\E \boxtimes \L_{f})$ gives the character 
\begin{displaymath}
\Ind^{G_r(\F_q)}_{T_r(\F_q)B_r^{r'}(\F_q)G_r^{r''}(\F_q)} (\chi \times (\psi\circ f)) : G_r(\F_q) \rightarrow\mathbb{C}.
\end{displaymath}
Now Corollary \ref{maincor} implies the following theorem.

\begin{theorem} Assume $p \geq r$ and $A_{r-1}$ is regular semisimple, and suppose the conditions above. Then we have
\begin{displaymath}
\Ind^{G_r(\F_q)}_{B_r(\F_q)} (\chi \times (\psi\circ f)) = \Ind^{G_r(\F_q)}_{T_r(\F_q)B_r^{r'}(\F_q)G_r^{r''}(\F_q)} (\chi \times (\psi\circ f)).
\end{displaymath}
\end{theorem}

\bibliographystyle{alpha}
\bibliography{./comparison_revised_fourth}

\end{document}